\documentclass[a4paper]{amsart}

\title[Compactness in abelian groups]{Compactness in abelian group theory}

\author[F.~Calderoni]{Filippo Calderoni}
\address{Department of Mathematics, Rutgers University, 
Hill Center for the Mathematical Sciences,
110 Frelinghuysen Rd.,
Piscataway, NJ 08854-8019}
\email{filippo.calderoni@rutgers.edu}

\author[A.~Ostrem]{Ava Ostrem}
\address{Department of Mathematics, Rutgers University, 
Hill Center for the Mathematical Sciences,
110 Frelinghuysen Rd.,
Piscataway, NJ 08854-8019}
\email{amo160@scarletmail.rutgers.edu}

\date{\today} 


\thanks{Both authors were supported by the NSF Grant DMS -- 2348819.}


\usepackage[T1]{fontenc}
\usepackage[osf]{mathpazo}

\usepackage{eucal}

\usepackage{microtype}

\usepackage[colorlinks=true,linkcolor=cyan,citecolor=cyan]{hyperref}

\usepackage{amssymb,
enumitem,
mathrsfs,
mathtools,
tikz,
upgreek,
verbatim,
hyphenat,
url,
amsmath
}
\usepackage[T1]{fontenc}
\usepackage[shadow]{todonotes}

\newenvironment{enumerate-(a)}{\begin{enumerate}[label={\upshape (\alph*)}, leftmargin=2pc]}{\end{enumerate}}
\newenvironment{enumerate-(a)-r}{\begin{enumerate}[label={\upshape (\alph*)}, leftmargin=2pc,resume]}{\end{enumerate}}
\newenvironment{enumerate-(a)-5}{\begin{enumerate}[label={\upshape (\alph*)}, leftmargin=2pc,start=5]}{\end{enumerate}}
\newenvironment{enumerate-(A)}{\begin{enumerate}[label={\upshape (\Alph*)}, leftmargin=2pc]}{\end{enumerate}}
\newenvironment{enumerate-(A)-r}{\begin{enumerate}[label={\upshape (\Alph*)}, leftmargin=2pc,resume]}{\end{enumerate}}
\newenvironment{enumerate-(i)}{\begin{enumerate}[label={\upshape (\roman*)}, leftmargin=2pc]}{\end{enumerate}}
\newenvironment{enumerate-(i)-r}{\begin{enumerate}[label={\upshape (\roman*)}, leftmargin=2pc,resume]}{\end{enumerate}}
\newenvironment{enumerate-(I)}{\begin{enumerate}[label={\upshape (\Roman*)}, leftmargin=2pc]}{\end{enumerate}}
\newenvironment{enumerate-(I)-r}{\begin{enumerate}[label={\upshape (\Roman*)}, leftmargin=2pc,resume]}{\end{enumerate}}
\newenvironment{enumerate-(1)}{\begin{enumerate}[label={\upshape (\arabic*)}, leftmargin=2pc]}{\end{enumerate}}
\newenvironment{enumerate-(1)-r}{\begin{enumerate}[label={\upshape (\arabic*)}, leftmargin=2pc,resume]}{\end{enumerate}}

\newtheorem{theorem}{Theorem}[section]
\newtheorem{lemma}[theorem]{Lemma}

\newtheorem{proposition}[theorem]{Proposition}

\newtheorem{fact}[theorem]{Fact}
\newtheorem{claim}{Claim}[theorem]
\theoremstyle{definition}
\newtheorem{definition}[theorem]{Definition}

\theoremstyle{remark}
\newtheorem{remark}[theorem]{Remark}

\numberwithin{equation}{section}

\DeclareMathOperator{\Reg}{Reg}

\begin{document}

\begin{abstract}
In this paper we discuss large cardinals and compactness theorems in abelian group theory.
More specifically, we generalize two classical compactness results for free abelian groups to the broader context of direct sums of cyclic groups.
\end{abstract}

\maketitle

\section{Introduction}

Large cardinal axioms postulate the existence of combinatorial properties of infinity. These axioms are far beyond the usual axioms of mathematics and expand the burden of the mathematical universe. A well-studied consequence of large cardinal axioms is compactness, the extent to which mathematical structures are determined by their local behavior.

The program of exploring the interactions between large cardinals and compactness dates back to as early as the influential work of Magidor~\cite{Mag71} and Shelah~\cite{She75} in the '70s. Since then, it has been a trend in set theory. Moreover, compactness and other set-theoretical methods have been proven useful in the subjects of infinite group theory and almost-free modules~\cite{BagMag,Cor,CorPov,CPT, MagShe, KanKos}.

In this manuscript, we focus on two classical compactness theorems for free abelian groups:

\begin{enumerate-(1)}
\item 
If \(\kappa\) is \(\omega_1\)-strongly compact, then every \(\kappa\)-free abelian group is free.
\item 
If \(\kappa\) is weakly compact, then every \(\kappa\)-free abelian group of cardinality \(\kappa\) is free.
\end{enumerate-(1)}

Generalizations of these two theorems to the context of free modules are already discussed in the comprehensive monograph of Eklof and Mekler~\cite{EklMek}. In this paper, instead, we discuss another generalization of free groups: the direct sums of cyclic groups. Although these groups have been already explored by Fuchs and Kaplansky in their famous books \cite{Fuc58} and \cite{Kap54}, it was only recently that they were referred to as \emph{\(\Sigma\)-cyclic groups} in~\cite{Fuc15}. Free abelian groups and \(\Sigma\)-cyclic groups present many similarities. However, compactness theorems for free groups may not generalize to the \(\Sigma\)-cyclic context, since the class of \(\Sigma\)-cyclic groups is strictly larger.

In this paper we prove the following two compactness theorems:

\begin{theorem}
\label{thm : comp Sigma cyclic}
    If \(\kappa\) is \(\omega_1\)-strongly compact, then every \(\kappa\)-\(\Sigma\)-cyclic abelian group is \(\Sigma\)-cyclic.
\end{theorem}

\begin{theorem}
\label{thm : weak comp Sigma cyclic}
     If \(\kappa\) is weakly compact, then every \(\kappa\)-\(\Sigma\)-free abelian group of cardinality \(\kappa\) is \(\Sigma\)-cyclic.
\end{theorem}

Our methods are the same as the ones employed in~\cite{EklMek}. Theorem~\ref{thm : comp Sigma cyclic} is proved using ultraproducts, while Theorem~\ref{thm : weak comp Sigma cyclic} follows from stationary reflection and a characterization of \(\Sigma\)-cyclic groups in terms of filtrations (Theorem~\ref{thm:Gamma invariant groups}).
Moreover, Theorem~\ref{thm : weak comp Sigma cyclic} requires a subtler analysis about purity and direct sums.

\subsection*{Acknowledgments}
We are grateful to Samuel Corson for pointing out inaccuracies in a previous version and for helpful remarks that improved the exposition of this paper. We also thank Tom Benhamou, Sean Cox and Alejandro Poveda for useful comments. The second author participated in the 2025 REU Program coordinated by the 
Center for Discrete Mathematics and Theoretical Computer Science (DIMACS) in 2025. We thank DIMACS for providing resources and a stimulating environment to conduct this research.

\section{Background}
\label{sec : background}

Let \(A\) be an abelian group and \(\{B_i:i\in I\}\) be a family of subgroups of \(A\). We define their \emph{sum} \(\sum_{i\in I} B_i\) as the smallest subgroup of \(A\) containing them. The group \(\sum_{i\in I} B_i\) can be explicitly described as the set of all finite sums of elements from the subgroups \(B_i\).

Further, we say that \(A\) is the \emph{direct sum} of its subgroups \(B_i\), with \(i\in I\) if:
\begin{enumerate-(i)}
\item 
\(A = \sum_{i\in I} B_i\);
\item 
For all \(i \in I\), we have \(B_i\cap \sum_{j\neq i} B_j =\{0\}\).
\end{enumerate-(i)}
In which case, we write \[A = \bigoplus_{i\in I} B_i\] or 
\(A = B_1 \oplus \dotsb \oplus B_n\),
whenever \(I=\{1,\dotsc, n\}\).

An abelian group \(A\) is \emph{free} if it is the direct sum of infinite cyclic groups. If these cyclic groups are generated by \(x_i\), \(i\in I\), we have 
    \[
    A = \bigoplus_{i\in I}\langle x_i\rangle.
    \]
    The set \(\{x_i: i\in I\}\) is a \emph{basis} of \(A\). The elements of \(A\) are linear combinations
    \[
    a = n_1x_1+\dotsb + n_k x_k \qquad k\geq 0, n_i\in\mathbb{Z}
    \]
    with different \(x_i\) and \(n_i\neq 0\).

A larger class of abelian groups are the $\Sigma$-cyclic groups. A group is \emph{$\Sigma$-cyclic} if it is the direct sum of cyclic groups. Notice that free groups coincide with torsion-free $\Sigma$-cyclic groups; however, a $\Sigma$-cyclic group can have torsion elements. 

It is known that subgroups of free groups are free. In a similar fashion, we have the following theorem due to Kulikov~\cite{Kul45}.

\begin{theorem}[Kulikov, 1945]
\label{thm:Kulikov}
  Subgroups of \(\Sigma\)-cyclic groups are 
    \(\Sigma\)-cyclic.
    \end{theorem}

An abelian group $A$ is \emph{$\kappa$-free} if all subgroups of $G$ with cardinality less than $\kappa$ are free. Similarly, an abelian group $A$ is \emph{$\kappa$-$\Sigma$-cyclic} if all subgroups of $G$ with cardinality less than $\kappa$ are $\Sigma$-cyclic.

Not all abelian groups are \(\Sigma\)-cyclic. For any prime \(p\), the Pr\"ufer $p$-group $\mathbb{Z}(p^{\infty})$ provides an interesting example. Recall that we can construct the countable infinite group $\mathbb{Z}(p^{\infty})$ as $\mathbb{Z}[\frac{1}{p},\frac{1}{p^2},\frac{1}{p^3},\dotsc]/\mathbb{Z}$. 
Any proper subgroup of $\mathbb{Z}(p^{\infty})$ is cyclic of order \(p^n\), for some \(n\in\mathbb{N}\). Since $\mathbb{Z}(p^{\infty})$ is infinite and its proper subgroups are finite, $\mathbb{Z}(p^{\infty})$ cannot be written as a direct sum of two proper subgroups. Since $\mathbb{Z}(p^{\infty})$ is not itself cyclic, it cannot be $\Sigma$-cyclic.

\section{Strong compactness}

In this section we will use an ultraproduct construction to prove Theorem~\ref{thm : comp Sigma cyclic}.
Recall that a \emph{filter}
$\mathcal{F}$ on a given set $I \neq \varnothing$ is a subset $\mathcal{F} \subseteq \mathcal{P}(I)$ such that:
\begin{enumerate-(1)-r}
    \item $I \in \mathcal{F}$ and $\varnothing \notin \mathcal{F}$;
    \item If $A,B \in \mathcal{F}$; then $A \cap B \in \mathcal{F}$;
    \item If $A, B \subseteq I$, $A \in \mathcal{F}$, and $B \supseteq A$, then $B \in \mathcal{F}$.
\end{enumerate-(1)-r}
If moreover,  for all $A \subseteq I$ either $A \in \mathcal{F}$ or $A^c \in \mathcal{F}$, we say that \(\mathcal{F}\) is an \emph{ultrafilter}.

If \(\kappa\) is a regular uncountable cardinal, and \(\mathcal{F}\) is a filter on \(I\), then \(\mathcal{F}\) is called \emph{\(\kappa\)-complete} if \(\mathcal{F}\) is closed under intersection of fewer than \(\kappa\) sets.

A cardinal $\kappa$ is \emph{$\omega_1$-strongly compact} if $\kappa$ is uncountable and for every set I, all $\kappa$-complete filters on $I$ extend to $\omega_1$-complete ultrafilters.

Suppose we have an infinite family of groups $\{G_i \mid i \in I\}$ and an ultrafilter $\mathcal{U}$ on $I$. Define an equivalence relation $\sim_\mathcal{U}$ on $\prod_{i \in I} G_i$ by $f \sim_\mathcal{U} g$ if and only if $\{i \in I \mid f(i) = g(i)\} \in \mathcal{U}$. Then, the \emph{ultraproduct} $\prod_{i \in I} G_i / \mathcal{U}$ is defined to be $\prod_{i \in I} G_i / {\sim_\mathcal{U}}$.
For every \(a\in \prod_{i \in I} G_i \), we denote by \(a_\mathcal{U}\) its equivalence class in \(\prod_{i \in I} G_i / \mathcal{U}\).
Moreover, we denote by \(0_\mathcal{U}\) the identity element in 
$\prod_{i \in I} G_i / \mathcal{U}$. 
We stress that for any \(a_\mathcal{U} \in \prod_{i \in I} G_i / \mathcal{U}\), we have
\[
a_\mathcal{U} = 0_\mathcal{U}\quad \iff \quad\{i\in I \mid a(i)= 0\text{ in \(G_i\)} \}\in\mathcal{U}.
\]

\begin{definition}
    We say that $B=\{b_i:i\in I\}$ is a set of \emph{independent cyclic generators} for the abelian group $A$ if $A = \sum_{i\in I} \langle b_i\rangle$ and for each $b_i \in B$, $\langle b_i \rangle \cap \sum_{j \neq i} \langle b_j \rangle = \{0\}$.
\end{definition}

If $B_i \subseteq G_i$ for every $i \in I$, then we denote by $\prod_{i \in I} B_i / \mathcal{U}$ the set \(\{b_\mathcal{U} \in \prod_{i \in I} G_i / \mathcal{U} : b(i) \in B_i\}\).

\begin{lemma}
    \label{lem:ultraproduct-cyclic-groups}
    Suppose we have a family of abelian groups $\{G_i : i \in I\}$ and an $\omega_1$-complete ultrafilter $\mathcal{U}$, and suppose $B_i$ is a set of independent cyclic generators for each $G_i$. Then, $\prod_{i \in I} B_i / \mathcal{U}$ is a set of independent cyclic generators for the ultraproduct \(\prod_{i \in I} G_i / \mathcal{U}\).
\end{lemma}

\begin{proof}
For the ease of exposition, let \(\mathcal{G} = \prod_{i \in I} G_i / \mathcal{U}\) and \(B = \prod_{i \in I} B_i / \mathcal{U}\).
    First, we will show that for any $x_{\mathcal{U}} \in B$, we have $\langle x_{\mathcal{U}} \rangle \cap \sum_{b\in B\setminus\{ x_\mathcal{U}\}} \langle b \rangle = \{0\}$. So, pick any $(x_0)_{\mathcal{U}} \in B$. Suppose that there are integers $k_0, k_1, \dotsc, k_n$ and distinct $(x_1)_{\mathcal{U}}, \dotsc, (x_n)_{\mathcal{U}} \in B \setminus \{(x_0)_{\mathcal{U}}\}$ such that $-k_0 (x_0)_U = k_1 (x_1)_{\mathcal{U}} + \cdots + k_n(x_n)_{\mathcal{U}}$. Then,
    \[
        \sum_{j=0}^{n} k_j (x_j)_{\mathcal{U}} = 0.
    \]
    This means that
    \[
        Y := \left\{ i \in I \mid \sum_{j=0}^{n} k_j x_j(i) = 0 \right\} \in \mathcal{U}.
    \]
    Consider
    \[
        D_{j,k} := \left\{ i \in I \mid x_j(i) \neq x_k(i) \right\}.
    \]
    Each $D_{j,k}$ is in $\mathcal{U}$ since $(x_j)_{\mathcal{U}} \neq (x_k)_{\mathcal{U}}$ for $j \neq k$. Now, the set
    \[
        W := Y \cap \bigcap_{0 \leq j < k \leq n} D_{j,k}
    \]
    belongs to the ultrafilter $\mathcal{U}$ because it is the finite intersection of sets in $\mathcal{U}$. Now, for any $i \in W$, we have $\sum_{j=0}^{n} k_j x_j(i) = 0$ and $x_j(i) \neq x_k(i)$ for $j \neq k$. Therefore, $-k_0 x_0(i) \in \langle B_i \setminus \{x_0(i)\} \rangle$. Since $x_0(i), \dotsc, x_n(i)$ are in $B_i$, and $B_i$ are a set of independent cyclic generators, it must be that $-k_0 x_0 (i) = 0$ for $i \in W$. Since $W \in \mathcal{U}$, it follows that $-k_0 (x_0)_{\mathcal{U}} = 0$. So, we proved that $\langle (x_0)_{\mathcal{U}} \rangle \cap \sum_{b\in B\setminus\{ x_\mathcal{U}\}} \langle b \rangle = \{0\}$.

    Next, we need to show that $\mathcal{G} = \sum_{b\in B} \langle b \rangle$. Pick any $y_\mathcal{U} \in \mathcal{G}$. For any $k_1, \dots, k_n \in \mathbb{Z}$, let 
    \[
        Y(k_1, \dots, k_n) = \left\{ i \in I \mid \exists x_0(i), \dots, x_n(i) \in B_i \quad   y(i) = \sum^n_{j =1} k_j x_j(i) \right\}
    \]

    Since each $B_i$ generates $G_i$,
    \[
    \bigcup_{n \in \omega \text{ and } k_1, \dotsc, k_n \in \mathbb{Z}} Y(k_1,\dotsc,k_n) = I
    \]
    Since there are countably many sets in this union and $\mathcal{U}$ is $\omega_1$-complete, there must be some \(n\in\omega\) and some \(n\)-tuple \(k_1,\dotsc,k_n\) such that $Y(k_1,\dotsc,k_n) \in \mathcal{U}$. For each $i\in Y(k_1,\dotsc,k_n)$, we define $x_1(i), \dotsc, x_n(i)\in B_i$ so that $y(i) = k_1 x_1(i) + \dotsb + k_n x_n(i)$. For $i \notin Y(k_1,\dots,k_n)$, we can set $x_1(i) = \cdots = x_n(i) = 0$. Then, $y_{\mathcal{U}} = \sum_{1 \leq j \leq n} k_i (x_i)_{\mathcal{U}}$. By construction, \((x_i)_\mathcal{U} \in B = \prod_{i \in I} B_i / \mathcal{U}\), so we proved that $B$ generates $\mathcal{G}$ as desired.
\end{proof}

With this lemma, we are able to prove the following compactness result.

\begin{proof}[Proof of Theorem~\ref{thm : comp Sigma cyclic}]
Suppose that \(\kappa\) is an \(\omega_1\)-strongly compact cardinal. Let \(A\) be a \(\kappa\)-\(\Sigma\)-cyclic group. As usual we denote by \(\mathcal{P}_\kappa(A)\) the set of subsets of \(A\) with cardinality \({<}\kappa\).
For any $Y \in \mathcal{P}_\kappa(A)$, let $U_Y = \{X \in \mathcal{P}_\kappa(A) \mid Y \subseteq X\}$. Then let
\[F_{\kappa}(A) = \{Z \subseteq \mathcal{P}_\kappa(A) \mid Z \supseteq U_Y \text{ for some } Y \in \mathcal{P}_{\kappa}(A) \}.\]
It is routine to check that $F_{\kappa}(A)$ is a $\kappa$-complete filter on $A$.
    
    For each $Y \in \mathcal{P}_\kappa(A)$, note that $\langle Y \rangle$ is a subgroup of $A$ with cardinality $<\kappa$, so $\langle Y \rangle$ is $\Sigma$-cyclic. Then, since $\kappa$ is $\omega_1$-strongly compact, $F_{\kappa}(A)$ is contained in an $\omega_1$-complete ultrafilter $\mathcal{U}$. Define the ultraproduct
    \[
    \mathcal{G} = \prod_{Y \in \mathcal{P}_\kappa(A)} \langle Y \rangle / \mathcal{U}
    \]
    By Lemma~\ref{lem:ultraproduct-cyclic-groups}, $\mathcal{G}$ is a $\Sigma$-cyclic group. So, it suffices to show that \(A\) embeds in \(\mathcal{G}\). We can recover a canonical embedding as follows. For any $a \in A$, define $\bar{a} \colon \mathcal{P}_\kappa(A) \to A$ by setting
    \[ \bar{a}(Y) = 
    \begin{cases}
        a & a \in Y. \\
        0 & a \notin Y.
    \end{cases}
    \]

Now, define $\varphi \colon A \to \mathcal{G}$ by $\varphi(a) = \bar{a}_\mathcal{U}$. We claim that $\varphi$ is an injective homomorphism.

For any $a,b \in A$, consider $U_{\{a,b,a+b\}}$, which is in $\mathcal{U}$. For any $Y \in U_{\{a,b,a+b\}}$, $\bar{a}(Y) + \bar{b}(Y)= a+b = \overline{a+b}(Y)$. Therefore, $\varphi(a)+\varphi(b) = \varphi(a+b)$ showing that \(\varphi\) is a group homomorphism. Moreover, for any nonzero $a \in A$, $U_{\{a\}}$ is in $\mathcal{U}$. For each $Y \in U_{\{a\}}$, $\bar{a}(Y) = a \neq 0$. Therefore, $(\bar{a})_{\mathcal{U}} \neq (\bar{0})_{\mathcal{U}}$, so $\varphi$ is injective.

Since $\varphi$ is an injective homomorphism from $A$ to $\mathcal{G}$, we obtain that $A$ must be isomorphic to a subgroup of $\mathcal{G}$. Since all subgroups of $\Sigma$-cyclic groups are $\Sigma$-cyclic by Kulikov's theorem (cf. Theorem~\ref{thm:Kulikov}), we conclude that $A$ is $\Sigma$-cyclic.
\end{proof}

\section{Weak compactness}

A cardinal \(\kappa\) is \emph{weakly compact} if it satisfies the partition property \(\kappa\to (\kappa)^2_2\).
There are a number of equivalent definitions for weakly compact cardinals, and a thorough introduction to weak compactness can be found in ~\cite[Chapter~9]{Jec}.
A crucial property of weakly compact cardinals is stationary reflection. We isolate a particular instance of stationary reflection below --- this is the only property of weakly compact cardinals that we use in our arguments.
\begin{lemma}
    \label{lem:reflection}
        Let \(\kappa\) be a weakly compact cardinal and let \(\{S_\alpha: \alpha<\kappa\}\) be a family of stationary subsets of \(\kappa\). Then there is a stationary \(T\subseteq \kappa \cap \Reg(\kappa)\) such that for all \(\lambda \in T\) and \(\alpha<\lambda\), the set
        \( S_\alpha \cap \lambda\) is stationary in \(\lambda\).
    \end{lemma}

In order to analyze groups of uncountable size, we use filtrations.

\begin{definition}
    Let \(A\) be an infinite abelian group with \(|A| = \kappa\). A \emph{\(\kappa\)-filtration} for \(A\) is a sequence \(\{ A_\alpha : \alpha<\kappa\}\) of subgroups of \(A\) such that:
    \begin{enumerate-(i)}
        \item 
        \(|A_\alpha| <\kappa\) for all \(\alpha<\kappa\);
        \item 
         \(A_\alpha\subseteq A_\beta\) whenever \(\alpha\leq \beta\);
        \item 
        if \(\gamma\) is a limit ordinal, then \(A_\gamma = \bigcup_{\alpha<\gamma}A_\alpha\).
        \item 
        \(A= \bigcup_{\alpha<\kappa}A_\alpha\).
    \end{enumerate-(i)}
\end{definition}

We introduce a crucial property for subgroups introduced by Pr\"ufer \cite{Pru23}.

\begin{definition}
    A subgroup $G$ of a group $A$ is called \emph{pure} if every equation \[nx = g \qquad (n\in\mathbb{N} , g\in G)\] has a solution in \(G\) whenever it has a solution in $A$.
\end{definition}

It is immediate to show that pureness is closed under taking unions. (See, e.g., \cite[p.~51]{Fuc15}.)

\begin{fact}
    \label{fac : pure union}
    The union of a chain of pure subgroups of \(G\) is also a pure subgroup of \(G\).
\end{fact}

Moreover, we recall the following classical result about pure subgroups.

\begin{theorem}[Szele]
    \label{thm:szele-pure}
    Every infinite subgroup \(H\) of a group $G$ can be embedded in a pure subgroup \(H_*\) of $G$ of the same cardinality.
\end{theorem}

We call \(H_*\) the \emph{purification} of \(H\). We do not discuss the construction of \(H_*\) in this paper --- we refer the reader to the monograph of Jacoby and Loth~\cite{JacLot} for full details.

 As a consequence of Fact~\ref{fac : pure union} and Szele's theorem, we have the following characterization of \(\Sigma\)-cyclic in terms of filtrations.

\begin{proposition}\label{prop:cyclic-filtration}
    Let \(\kappa\) be a regular cardinal and \(A\) be an abelian group with \(|A| = \kappa\).
     Then \(A\) is \(\kappa\)-\(\Sigma\)-cyclic if and only if \(A\) has a filtration consisting of \(\Sigma\)-cyclic  (pure)
     subgroups.
\end{proposition}

\begin{proof}
    \((\Rightarrow)\) Suppose $A$ is $\kappa$-$\Sigma$-cyclic. Let $\{a_i : i < \kappa\}$ be a generating set for $A$, so $A = \langle a_i : i < \alpha \rangle$.
    We inductively construct a filtration $\{ A_{\alpha} : i < \alpha\}$ consisting of pure subgroups of \(A\). Let \(A_0\)  be the countable group \(\langle a_0\rangle_*\). Then, let \(A_{\alpha+1} = \langle A_\alpha \cup \{a_{\alpha+1}\}\rangle_*\). For \(\lambda\) limit, we define \(A_\lambda = \bigcup_{\alpha<\lambda} A_\alpha\). 
     By Fact~\ref{fac : pure union}, the union of a chain of pure subgroups of \(A\) is also a pure subgroup of \(A\), so each set in the filtration is a pure subgroup of $A$.
    Since each $A_{\alpha}$ is $<\!\kappa$-generated, each $A_{\alpha}$ has cardinality $<\!\kappa$. Since $A$ is $\kappa$-$\Sigma$-cyclic, each $A_{\alpha}$ is $\Sigma$-cyclic. 

    \((\Leftarrow)\) Suppose $A = \cup_{\alpha < \kappa} A_{\alpha}$ is a filtration of $A$ and each $A_{\alpha}$ is $\Sigma$-cyclic. Suppose $G$ is a subgroup of $A$ with \(|G|=\lambda<\kappa\). Then, $G \subseteq A_{\gamma}$ for some $\gamma < \kappa$.
    To see this, consider an enumeration \(\{g_i:i<\lambda\}\) of \(G\).
    For each \(i<\lambda\), let \(\alpha_i = \min \{\alpha<\kappa: g_i \in A_{\alpha}\}\). Then set \(\gamma = \sup_{i<\lambda} \alpha_i\), which is smaller than \(\kappa\), because \(\kappa\) is regular. Since each $g_i \in G$ is in $A_{\gamma}$, $G$ is a subgroup of $A_{\gamma}$.
    So, $G$ is $\Sigma$-cyclic by Theorem~\ref{thm:Kulikov}.
\end{proof}

\begin{remark}
    Purity of the groups of the filtration \(\{A_\alpha:\alpha<\kappa\}\) will be a crucial property to characterize \(\Sigma\)-cyclic groups in terms of their \(\Gamma\)-invariant. Since any two filtrations agree on a club it is possible to argue that every filtration consists of pure subgroups on a club set.
\end{remark}

\subsection{The \(\Gamma\)-invariant for \(\Sigma\)-cyclic groups}

For \(X,Y\subseteq \kappa\), define
    \[
    X\sim Y \iff \text{there is a club \(C\subseteq \kappa\) such that \(X\cap C = Y\cap C\)}.
    \]
    It is routine to check that \(\sim\) is an equivalence relation on \(\mathcal{P}(\kappa)\).

    \begin{definition}
        Let \(A\) be an abelian group with \(|A|=\kappa\) and \(\{A_\alpha: \alpha<\kappa\}\) be a \(\kappa\)-filtration for \(A\). Define
        \[
        \Gamma(A) = \{ \alpha<\kappa : A/A_\alpha\text{ is not \(\kappa\)-\(\Sigma\)-cyclic}\} /{\sim}.
        \]
    \end{definition}

    \begin{remark}
        Note that     \[\Gamma(A) = \{ \alpha<\kappa :
    \{ \beta>\alpha : \beta<\kappa \text{ and } A_\beta/A_\alpha\text{ is not \(\Sigma\)-cyclic}\} \text{ is stationary in \(\kappa\)}
    \}/\sim.\]

    In fact, for any \(\alpha<\kappa\)
    the quotient group
    \(A/A_\alpha\) is not $\kappa$-$\Sigma$-cyclic precisely when
    the set \(\{\beta>\alpha :\beta<\kappa \text{ and } A_\beta/A_\alpha \text{ is not \(\Sigma\)-cyclic}\}\) is stationary in \(\kappa\).
    
    To see this, assume first that \(A/A_\alpha\) is not \(\kappa\)-\(\Sigma\)-cyclic.
    Then, Proposition~\ref{prop:cyclic-filtration} yields that \(A/A_\alpha\) does not have any filtration consisting of \(\Sigma\)-cyclic groups. On the other hand, note that \(\{A_\beta/A_\alpha: \beta>\alpha\}\) is a filtration for \(A/A_\alpha\). Since any two filtrations of a group must agree on a club, it follows that for every club \(C\subseteq \kappa\) there is some \(\beta\in C\cap (\alpha,\kappa)\) such that \(A_\beta/ A_\alpha\) is not \(\Sigma\)-cyclic. This shows that the set \(\{\beta>\alpha :\beta<\kappa \text{ and } A_\beta/A_\alpha \text{ is not \(\Sigma\)-cyclic}\}\) is stationary.

    If for \(\alpha<\kappa\) the set \(\{\beta>\alpha :\beta<\kappa \text{ and } A_\beta/A_\alpha \text{ is not \(\Sigma\)-cyclic}\}\) is stationary, then for every club \(C\subseteq \kappa\) there is some \(\gamma\in C\cap (\alpha,\kappa)\) such that \(A_\gamma/A_\alpha\) is not \(\Sigma\)-cyclic. It follows that \(A/A_\alpha\) does not admit a filtration consisting of \(\Sigma\)-cyclic, and thus \(A/A_\alpha\) is not \(\kappa\)-\(\Sigma\)-cyclic by Proposition~\ref{prop:cyclic-filtration}.
    \end{remark}

    Next, we discuss a characterization of \(\Sigma\)-cyclic groups in terms of set-theoretical properties of their \(\Gamma\)-invariant. In the proof, we iteratively apply the following fact about \(\Sigma\)-cyclic groups.

\begin{theorem}
    \label{thm:cyclic-quotient-summand}
    Let $A$ be an abelian group and $B$ a pure subgroup of $A$ such that $A/B$ is $\Sigma$-cyclic. Then, $B$ is a direct summand of $A$ and $A \cong B \oplus A/B$.
\end{theorem}

\begin{proof}
    See \cite[p. 15]{Kap54}.
\end{proof}

    \begin{theorem}
        \label{thm:Gamma invariant groups}
        Let \(A\) be \(\kappa\)-\(\Sigma\)-cyclic with \(|A|=\kappa\). Then \(A\) is \(\Sigma\)-cyclic if and only\footnote{We denote by \(0\) the equivalence class \([\emptyset]_\sim\).} if \(\Gamma(A) = 0\).
    \end{theorem}
    \begin{proof}
        \((\Rightarrow)\) Let \( A = \bigoplus_{i<\kappa}\langle c_i\rangle\). Define a filtration for \(A\) by setting \(A_\alpha =\bigoplus_{i<\alpha}\langle c_i\rangle\) with the usual stipulation that \(A_0 = \{0\}\). Then, we have
        \[
        \{ \alpha<\kappa : A/A_\alpha \text{ is not \(\kappa\)-\(\Sigma\)-cyclic}\} = \emptyset
        \]
        because \(A/A_\alpha\) is isomorphic to the $\Sigma$-cyclic group $\bigoplus_{\alpha \leq i<\kappa}\langle c_i\rangle$.

        \((\Leftarrow)\) Suppose that \(\{A_\alpha: \alpha<\kappa\}\) is a filtration for \(A\) consisting of $\Sigma$-cyclic pure subgroups of $A$. Assume that \(\Gamma(A) = 0\). Thus,
        \[ E = \{\alpha<\kappa : A/A_\alpha \text{ is not \(\kappa\)-\(\Sigma\)-cyclic }\}\sim \emptyset,\]
        so there is a club \(C\subseteq \kappa\) such that \(C\cap E = C\cap \emptyset = \emptyset\). Let \(f\colon \kappa \to C \) be a continuous increasing function which enumerates \(C\). Then, for \(\alpha<\kappa\), define \(B_\alpha = A_{f(\alpha)}\). Note that \(\{ B_\alpha : \alpha<\kappa\}\) is also a filtration for \(A\) consisting of \(\Sigma\)-cyclic groups. Moreover, without loss of generality, we can assume that \(B_0 = \{0\}\). For all \(\alpha<\kappa\), note that
        \[
        B_{\alpha+1}/B_\alpha \subseteq A/B_\alpha = A/A_{f(\alpha)}.
        \]
        Since \(f(\alpha) \in C\), it follows that \(A/A_{f(\alpha)}\) is \(\kappa\)-\(\Sigma\)-cyclic, and therefore \(B_{\alpha+1}/B_\alpha\) is \(\Sigma\)-cyclic.
        We will prove that \(A =  \bigoplus_{\alpha<\kappa} B_{\alpha+1}/B_\alpha\) by induction, showing that \(A\) is $\Sigma$-cyclic.
        Suppose $B_{\beta} \cong \bigoplus_{\alpha<\beta} B_{\alpha+1}/B_\alpha$. By our previous observation, $B_{\beta+1}/B_{\beta}$ is $\Sigma$-cyclic. Moreover, since \(B_\alpha\) is pure in \(A\), it is clear that \(B_\alpha\) is pure in \(B_{\alpha+1}\). So, by Theorem~\ref{thm:cyclic-quotient-summand},
        \[
        B_{\beta+1} \cong B_{\beta}\oplus B_{\beta+1}/B_{\beta}
        \cong
        \bigoplus_{\alpha<\beta} B_{\alpha+1}/B_\alpha
        \oplus
        B_{\beta+1}/B_{\beta}  
        \cong
        \bigoplus_{\alpha<\beta+1} B_{\alpha+1}/B_\alpha.
        \]
        
        Now, suppose $\beta < \kappa$ is a limit cardinal and for all $\gamma < \beta$, suppose that $B_{\gamma} \cong\bigoplus_{\alpha<\gamma} B_{\alpha+1}/B_\alpha$. Since $\{B_{\alpha}:\alpha <\kappa\}$ is a filtration, we have continuity at limit levels. Therefore, 
        \[             
            B_{\beta} =  \bigcup_{\alpha < \beta} B_{\alpha} \cong \bigcup_{\gamma < \beta} \bigoplus_{\alpha<\gamma} B_{\alpha+1}/B_\alpha \cong  \bigoplus_{\alpha < \beta} B_{\alpha+1}/B_\alpha.
        \] 

In conclusion, we have \(A = \bigcup_{\alpha<\kappa} B_\alpha = \bigoplus_{\alpha < \kappa}B_{\alpha+1}/B_\alpha\) as desired.
\end{proof}

\begin{theorem}
        \label{thm:compactness sigma cyclic}
        Let \(\kappa\) be weakly compact. If \(A\) is an abelian group with \(|A| = \kappa\) and \(A\) is \(\kappa\)-\(\Sigma\)-cyclic, then \(A\) is \(\Sigma\)-cyclic.
    \end{theorem}

    \begin{proof}[Proof of Theorem~\ref{thm : weak comp Sigma cyclic}]
    Let \(\kappa\) be weakly compact. Assume that \(A\) is an abelian group with \(|A| = \kappa\) and that \(A\) is \(\kappa\)-\(\Sigma\)-cyclic.

    Suppose \(A\) is not \(\Sigma\)-cyclic towards contradiction.
        Let \(\{A_\alpha: \alpha<\kappa\}\) be a filtration for \(A\) consisting of \(\Sigma\)-cyclic groups.
        Without loss of generality, assume that \(A_0 = \{0\}\) and \(|A_\alpha|\leq{|\alpha|+\aleph_0}\). Then let
        \[
        E = \{\alpha<\kappa :
        \{ \beta>\alpha : \beta<\kappa \text{ and } A_\beta/A_\alpha\text{ is not \(\Sigma\)-cyclic}\}
        \text{ is stationary in \(\kappa\)}.
        \}
        \]
        Since \(A\) is not \(\Sigma\)-cyclic, \(E\) is stationary by Theorem~\ref{thm:Gamma invariant groups}. Now we build a family \(\{S_\alpha\}\) of stationary subsets of \(\kappa\)  by setting
        \[
        S_\alpha =
        \begin{cases}
            \{\beta>\alpha : \beta<\kappa \text{ and } A_\beta/A_\alpha \text{ is not \(\Sigma\)-cyclic}\} &\alpha \in E.\\
            E &\text{otherwise}.
        \end{cases}
        \]
        Let \(T\subseteq \kappa \cap \Reg(\kappa)\) be a stationary set such that for all \(\lambda\in T\) and \(\alpha<\lambda\) the set \(S_\alpha\cap \lambda\) is stationary in \(\lambda\).

        Now pick any \(\lambda \in T\). First note that \(E\cap \lambda\) is stationary in \(\lambda\). This is because 
        \[\{\beta>0 : A_\beta / A_0\text{ is not \(\Sigma\)-cyclic}\} = \emptyset\] is clearly not stationary. Therefore, \(0 \notin E\) and consequently
        \(S_0\cap \lambda = E\cap \lambda\), which is stationary by the assumption on \(T\). Also, note that \(\{A_\alpha: \alpha<\lambda\}\) is a filtration for the \(\Sigma\)-cyclic group \(A_\lambda\).

        \begin{claim}For any \(\lambda \in T\), we have
            \[E\cap \lambda \subseteq \{ \alpha<\lambda : \{ \beta>\alpha: \beta<\lambda \text{ and \(A_\beta/A_\alpha\) is not \(\Sigma\)-cyclic}\} \text{ is stationary in } \lambda \}\] and both sets are stationary in $\lambda$.
        \end{claim}

        \begin{proof}
            For \(\alpha \in E\cap \lambda\), we know that \(S_\alpha\) reflects at \(\lambda\) by Lemma~\ref{lem:reflection} and the choice of \(T\). This means that the set
            \[S_\alpha \cap \lambda  = \{\beta > \alpha : \beta <\lambda \text{ and } A_\beta/A_\alpha \text{ is not \(\Sigma\)-cyclic}\}
            \]
            is stationary in \(\lambda\). Since \(E\cap \lambda\) is stationary, so is the larger set.
        \end{proof}

         Given the claim, \(\Gamma(A_\lambda)\neq 0\). This shows that \(A_\lambda\) is not \(\Sigma\)-cyclic due to Theorem~\ref{thm:Gamma invariant groups}, and this is a contradiction.
    \end{proof}

\bibliographystyle{alpha}

\end{document}